\documentclass[11pt,leqno]{amsart}
\usepackage{amsmath,amssymb,amsthm}
\usepackage{hyperref}

\usepackage{tikz-cd}

\usepackage{hyperref}

\usepackage{bbm}

\newcommand{\R}{\mathbb{R}}

\DeclareMathOperator{\co}{co}

\renewcommand{\geq}{\geqslant}
\renewcommand{\leq}{\leqslant}

\newcommand{\norm}[1]{\left\Vert#1\right\Vert}

\newcommand{\Lip}{{\mathrm{Lip}}_0}
\newcommand{\justLip}{{\mathrm{Lip}}}

\newcommand{\cco}{\overline{\operatorname{co}}}
\newcommand{\ext}[1]{\operatorname{ext}\left(#1\right)}

\newtheorem{theorem}{Theorem}[section]
\newtheorem{lemma}[theorem]{Lemma}

\newtheorem{proposition}[theorem]{Proposition}
\newtheorem{corollary}[theorem]{Corollary}
\theoremstyle{definition}
\newtheorem{definition}[theorem]{Definition}
\newtheorem{example}[theorem]{Example}

\theoremstyle{remark}
\newtheorem{remark}[theorem]{Remark}

\numberwithin{equation}{section}
\newcommand{\abs}[1]{\left\lvert#1\right\rvert}

\def\fnote#1{\footnote}

\def\ignora#1{}
\def\n3#1{\left\vert  \! \left\vert \! \left\vert \, #1 \, \right\vert \!
  \right\vert \! \right\vert }

\newcommand{\iten}{\ensuremath{\widehat{\otimes}_\varepsilon}}
\newcommand{\pten}{\ensuremath{\widehat{\otimes}_\pi}}

\renewcommand{\leq}{\le}

\usepackage{accents}
\usepackage[most]{tcolorbox}
\usepackage{comment}

\let\emptyset\varnothing

\newcommand{\N}{\mathbb{N}}

\newcommand{\U}{\mathcal{U}}

\newcommand{\eps}{\varepsilon}
\newcommand{\conv}{\text{co}}

\begin{document}

\author{ Esteban Martínez Vañó }\address{Universidad de Granada, Facultad de Ciencias. Departamento de An\'{a}lisis Matem\'{a}tico, 18071-Granada
(Spain)} \email{ emv@ugr.es}
\urladdr{\url{https://eleccionobarbarie.com}}

\author{ Abraham Rueda Zoca }\address{Universidad de Granada, Facultad de Ciencias. Departamento de An\'{a}lisis Matem\'{a}tico, 18071-Granada
(Spain)} \email{ abrahamrueda@ugr.es}
\urladdr{\url{https://arzenglish.wordpress.com}}



\title{The uniform strong diameter two property}

\begin{abstract}
We study a uniform version of the strong diameter two property. In particular, we find a characterisation that does not involve ultrafilters and we use it to provide some examples of spaces with this uniform property that do not follow from previously known results.
\end{abstract}

\maketitle

\section{Introduction}

The study of ultrapower techniques has attracted the attention of researches in many branches of functional analysis because of their usefulness in order to determine the local structure of metric and Banach spaces. As a matter of fact, observe that in \cite[Theorem 7.6]{beli00} ultraproduct techniques are used in order to establish a connection between the existence of (metric) retractions and the existence of (linear) projections on Banach spaces. Another use of ultraproduct techniques can be seen in \cite[Chapter 11]{alka2006}, where ultraproduct techniques are used in order to characterise the finite-representability of $\ell_1$ in a Banach space in terms of a condition on the type of the underlying space. A more recent application can be found in \cite{rz25-2}, where ultraproduct spaces where used in order to provide an explicit proof of the fact that $X\pten Y$ and $X\iten Y$ are superreflexive if, and only if, either $X$ or $Y$ is finite dimensional.

Due to its importance, a big effort has been devoted to describe the topological and geometrical structure of ultrapowers of Banach spaces. We refer the reader, for instance, to \cite{greraj23,hein80,james,stern78,Tu}, where a study of properties like being an $L_p$-space, the reflexivity of ultrapowers or the weak compactness of sets are studied. For geometric properties of Banach spaces in ultrapower spaces, we refer the reader to \cite{bksw,ggr22,hardtke18,kw04,rz25,talponen17} for study on properties like the extremal structure of the unit ball of ultrapower spaces, almost squareness, the Daugavet property or the slice diameter two property. 

Having a closer look to the above geometric properties of Banach spaces, one discovers that the complexity of its analysis in ultrapower spaces dramatically varies depending on whether or not the property involves the use of continuous linear functionals. This variation relies on a classical result about ultrapower spaces: given a Banach space $X$ and a free ultrafilter $\mathcal U$ over $\mathbb N$, it follows that $(X_\mathcal U)^*=(X^*)_\mathcal U$ if, and only if, $X$ is superreflexive. Moreover, if $X$ is not superreflexive, there is not a good description of the topological dual of $X_\mathcal U$. Because of this reason, informally speaking, properties of Banach spaces which are described using elements of the topological dual may be difficult to analyse in ultrapower spaces. This is the reason why, until \cite{rz25}, little was known about the diameter two properties in ultrapowers. Before this paper, ultrapower Banach spaces enjoying the diameter two properties were known by studying the inheritance by ultrapower spaces of certain properties implying the diameter two properties (this is the case, for instance, of the Daugavet property \cite{bksw} or almost squareness properties \cite{hardtke18}). However, in the above mentioned paper \cite{rz25} a complete characterisation of the slice-D2P in ultrapower spaces was obtained making use of the following result of Y. Ivakhno \cite[Lemma 1]{iva06}: a Banach space $X$ has the slice-D2P if, and only if, $B_X=\overline{\conv}\left\{\frac{x+y}{2}: x,y\in B_X, \Vert x-y\Vert\geq 2-\varepsilon\right\}$ holds for every $\varepsilon>0$. Thanks to the above result, which is a characterisation of the slice-D2P that avoids the access to the topological dual $X^*$, in \cite[Definition 3.5]{rz25} the \textit{uniform slice diameter two property (uniform slice-\emph{D2P})} is defined as a Banach space property on $X$ which determines the slice-D2P in all its ultrapowers $X_\mathcal U$ over any free ultrafilter $\mathcal U$ over $\mathbb N$. This characterisation permitted to prove in \cite[Section 3]{rz25} that all the classical Banach spaces which are known to enjoy the slice-D2P actually satisfy its uniform version. However, in \cite[Section 4]{rz25} it is proved that the slice-D2P is not equivalent to its uniform version. Indeed, based on an example of \cite{kw04}, where it is proved that the Daugavet property does not imply its uniform version, it is proved that for every $\varepsilon>0$, there exists a Banach space with the Daugavet property (and in particular with the slice-D2P) such that, for every free ultrafilter $\mathcal U$ over $\mathbb N$, the unit ball of $X_\mathcal U$ contains a slice of diameter smaller than $\varepsilon$ \cite[Theorem 4.6]{rz25}. 

The aim of the present note is to study a uniform version of the SD2P analogous to the cited ones for the Daugavet property or the slice-D2P. After introducing some notation and background on ultraproducts in Section \ref{sec:background}, our first step is, inspired by the strategy of \cite{rz25}, to find a characterisation of the SD2P which does not require to access the topological dual of the given space. We accomplish this in Theorem \ref{carsd2p} and use this result to obtain the desired characterisation in Theorem~\ref{ultrasd2p} together with further consequences (Corollary~\ref{charUSD2P}). The rest of Section \ref{sec:USD2P} is devoted to the study of some stability properties, and to present some examples which either follow from known characterisations of uniform versions of stronger properties than the SD2P or from known results on stability of ultrapowers. Finally, in Section \ref{sec:examples} we present a couple of examples which, unlike examples in Section \ref{sec:USD2P}, do not follow from known results and make a strong use of the characterisation given in Corollary~\ref{charUSD2P}.

\section{Background and notation}\label{sec:background}

We will consider real Banach spaces unless we state the contrary in particular examples. 

Given a Banach space $X$, $B_X$ (respectively $S_X$) stands for the closed unit ball (respectively the unit sphere) of $X$. We will denote by $X^*$ the topological dual of $X$. Given a subset $C$ of $X$, we will denote by $\conv(C)$ the convex hull of $C$. We also denote by $\conv_m(C)$ the set of all convex combinations of at most $m$ elements of $C$, that is
$$\conv_m(C):=\left\{\sum_{i=1}^m \lambda_i x_i: \lambda_1,\ldots, \lambda_m\in [0,1], \sum_{i=1}^m \lambda_i=1, x_1,\ldots, x_m\in C \right\}.$$

If $C$ is a bounded subset of a Banach space $X$, by a \textit{slice} of $C$ we will mean a set of the following form
$$S(C,f,\alpha):=\{x\in C:  f(x)>\sup f(C)-\alpha\}$$
where $f\in X^*$ and $\alpha>0$. Notice that a slice is nothing but the non-empty intersection of a half-space with the bounded (and not necessarily convex) set $C$.

Let us also introduce the relevant geometric properties we will discuss along the text:
\begin{definition}
Given a Banach space $X$, we say that
\begin{enumerate}
    \item $X$ has the \textit{slice diameter two property \emph{(}slice-\emph{D2P)}} if every slice of its unit ball has diameter exactly 2.
    \item $X$ has the \textit{diameter two property \emph{(D2P)}} if every non-empty relatively weakly open subset of its unit ball has diameter exactly 2.
    \item $X$ has the \textit{strong diameter two property \emph{(SD2P)}} if every convex combination of slices of $B_X$ has diameter 2.
\end{enumerate}
\end{definition}
It is a consequence of a lemma of Bourgain that the SD2P implies the D2P and its is obvious that the D2P implies the slice-D2P. Furthermore, it is also know that all of them are different properties (see, for example, \cite{blr15big} and\cite{blr15}). We refer the reader to \cite{ahltt16,almt21,blr15,lanru2020} for background about these properties. A relevant class of Banach spaces satisfying the SD2P are Banach spaces with the Daugavet property (see \cite[Theorem 3.2.1]{kmrzw25}). Observe that Banach spaces $X$ where all its ultrapower spaces enjoy the Daugavet property has been characterised (c.f. e.g. \cite[Theorem 7.6.4]{kmrzw25})

In the search of Theorem~\ref{ultrasd2p}, a reformulation of the SD2P given in \cite{lmr} is essential. We leave such result here for easy reference.

\begin{theorem}\label{theo:charSD2P2019}{\emph{(}\cite[Theorem 3.1]{lmr}\emph{)}}
Let $X$ be a Banach space. The following assertions are equivalent:
\begin{enumerate}
    \item $X$ has the \emph{SD2P}.
    \item Given any convex combination of slices $C$ of $B_X$ and given any $\varepsilon>0$ there exists an $x\in C$ with $\Vert x\Vert>1-\varepsilon$.
\end{enumerate}
\end{theorem}

We will make use of the above theorem in order to provide a description of the SD2P in ultrapower spaces. In order to do so, the following notation will be useful: given a Banach space $X$ we define
$$C_m^{n, \eps, \alpha}(X) := \co_m\left( \left\{z \in \ell_\infty^n(X): \norm{z}_\infty \leq \alpha \wedge \dfrac{1}{n}\norm{\sum_{i=1}^n z_i} > 1-\eps \right\} \right).$$
If $m = 1$ or $\alpha = 1$ we will omit the indices and if $\alpha = 1+\eps$ we will denote the previous set by $C_m^{n, \eps+}(X)$. Furthermore, we will also use the following notation along the text
$$D_m^{n, \eps,\alpha}(X) := \sup_{z \in B_{\ell_\infty^n}(X)} d\left( z, C_m^{n, \eps,\alpha} (X) \right)$$
with the same agreements for the indices of the sets.

We will also use in several parts of the text the notion of \textit{almost isometric ideal}: given $Y$ a subspace of a Banach space $X$, we say that $Y$ is an \emph{almost isometric ideal} (ai-ideal) in $X$ if $Y$ is locally complemented in $X$ by almost isometries.
This means that for each $\varepsilon>0$ and for each
finite-dimensional subspace $E\subseteq X$ there exists a linear
operator $T:E\to Y$ satisfying
\begin{enumerate}
\item\label{item:ai-1}
  $T(e)=e$ for each $e\in E\cap Y$, and
\item\label{item:ai-2}
  $(1-\varepsilon) \Vert e \Vert \leq \Vert T(e)\Vert\leq
  (1+\varepsilon) \Vert e \Vert$
  for each $e\in E$,
\end{enumerate}
i.e. $T$ is a $(1+\varepsilon)$-isometry fixing the elements of $E$.
If the $T$'s satisfy only (\ref{item:ai-1}) and the right-hand side of (\ref{item:ai-2}) we get the well-known
concept of $Y$ being an \emph{ideal} in $X$ \cite{GKS}.

Note that the Principle of Local Reflexivity means that $X$ is an ai-ideal in $X^{**}$ for every Banach space $X$. Moreover, the Daugavet property and all of the diameter two properties are inherited by ai-ideals (see \cite{aln2}).

Finally, we will briefly introduce ultraproduct spaces. Given a family of Banach spaces $\{X_\lambda: \lambda \in \Lambda\}$ we denote 
$$\ell_\infty(\Lambda,X_\lambda):=\left\{f \in \prod \limits_{\lambda\in \Lambda} X_\lambda: \sup_{\lambda \in \Lambda}\Vert f(\lambda)\Vert<\infty\right\},$$
and for any free ultrafilter $\mathcal U$ over $\Lambda$ we can consider the following closed subspace of $\ell_\infty(\Lambda, X_\lambda)$ 
$$c_{0,\mathcal U}(\Lambda,X_\lambda):= \left\{f\in \ell_\infty(\Lambda, X_\lambda): \lim_\mathcal U \Vert f(\lambda)\Vert=0 \right\}.$$
The \textit{ultraproduct of $\{X_\lambda: \lambda \in \Lambda\}$ with respect to $\mathcal U$} is the quotient Banach space
$$(X_\lambda)_\mathcal U:=\ell_\infty(\Lambda,X_\lambda)/c_{0,\mathcal U}(\Lambda,X_\lambda).$$
If every $X_\lambda$ is equal to some fix Banach space $X$ we will call the ultraproduct of that family the \textit{ultrapower of $X$}, which will be denoted by $X_\U$.

We denote the image of $x \in \ell_\infty(\Lambda,X_\lambda)$ by some $\lambda \in \Lambda$ as $x_\lambda$ and by $[x_\lambda]_\mathcal U$, or simply $[x_\lambda]$ if there is no confusion, the coset $x + c_{0,\mathcal U}(\Lambda, X_\lambda) \in (X_\lambda)_\U$. Furthermore, from the definition of the quotient norm, it is not difficult to prove that
$$\Vert [x_\lambda]\Vert_\U=\lim_\mathcal U \Vert x_\lambda\Vert$$
for every $[x_\lambda] \in (X_\lambda)_\mathcal U$.  This implies that the canonical inclusion $j:X\rightarrow X_\mathcal U$ given by
$$j(x):=[x]_\mathcal U$$
is an into linear isometry. It is well known (c.f. e.g. Propositions 6.1 and 6.2 in \cite{hein80}) that $X_\mathcal U$ is finitely representable in $j(X)$. However, from an inspection of the proofs and the operators used there it can be concluded that $j(X)$ is indeed an almost isometric ideal in $X_\mathcal U$.

Let us finish wit a useful remark for our future purposes.

\begin{remark}\label{remark:elereprebolaunidad}
Let $(X_\lambda)_\mathcal U$ be a ultraproduct space for some free ultrafilter $\mathcal U$ over an infinite set $\Lambda$. Observe that, given $x\in B_{(X_\lambda)_\mathcal U}$, we can always choose $x_\lambda\in B_{X_\lambda}$ such that $x=[x_\lambda]_\mathcal U$, that is, we can always select representatives where every element is in the unit ball. Indeed, if we select some representatives $x=[y_\lambda]_\mathcal U$ we have two possibilities:
\begin{enumerate}
    \item If $\Vert x\Vert_\mathcal U=\lim\limits_\mathcal U \Vert y_\lambda\Vert<1$, then define
    $$A:=\{\lambda\in\Lambda: \vert \Vert y_\lambda\Vert-\Vert x\Vert_\mathcal U\vert <1-\Vert x\Vert_\mathcal U\}\in\mathcal U.$$
    Observe that given $\lambda\in A$ it follows that $\Vert y_\lambda\Vert<1$, so if we define $x_\lambda=y_\lambda$ if $\lambda\in A$ and $x_\lambda=0$ if $\lambda\notin A$, then $x=[x_\lambda]_\mathcal U$ is the desired representation.
    \item If $\Vert x\Vert_\mathcal U=\lim\limits_\mathcal U \Vert y_\lambda\Vert=1$, then define
    $$A:=\{\lambda\in\Lambda: \vert \Vert y_\lambda\Vert-\Vert x\Vert_\mathcal U\vert <1\}\in\mathcal U.$$
    Observe that given $\lambda\in A$ it follows that $\Vert y_\lambda\Vert>0$, so if we define $x_\lambda=\frac{y_\lambda}{\Vert y_\lambda\Vert}$ if $\lambda\in A$ and $x_\lambda=0$ if $\lambda\notin A$, then $x=[x_\lambda]_\mathcal U$ is the desired representation.
\end{enumerate}
\end{remark}

\section{The uniform strong diameter two property}\label{sec:USD2P}

As stated in the introduction our objective is to study the following uniformization of the strong diameter two property:

\begin{definition}
Let $X$ be a Banach space. We say that $X$ has the \textit{uniform strong diameter two property} (USD2P) if for any free ultrafilter $\U$ over $\N$ the ultrapower $X_\U$ has the SD2P.
\end{definition}

Observe that if $X$ has the USD2P, then $X$ also has the SD2P because $X$ is an almost isometric ideal in any of its ultrapowers and \cite[Proposition 3.3]{aln2} asserts that the SD2P is preserved by ai-ideals. In fact, it is a stronger property because in \cite[Theorem 4.6]{rz25} it is proved that for any $\eps > 0$ there exists a Banach space with the Daugavet property (so with the SD2P) such that, for any free ultrafilter $\U$ over $\N$, the unit ball of $X_\U$ contains slices of diameter less than $\eps$.

Despite that, we can see that classical Banach spaces which are known to enjoy the SD2P actually have the uniform version.

\begin{example}
ASQ spaces have the USD2P by \cite[Corollary 2.2]{blr14}, \cite[Proposition 2.5]{all16} and \cite[Proposition 4.2]{hardtke18}. In particular, non-reflexive M-embedded spaces have the USD2P \cite[Corollary 4.3]{all16}.
\end{example}

In the next example we aim to prove that for $L_1$ spaces the USD2P and the SD2P are equivalent, but to do so we need a little bit of preliminaries.

Recall that given a measure space $(\Omega, \Sigma,\mu)$, $A \in \Sigma$ is called an \emph{atom} for $\mu$ if $\mu(A)>0$ and given any measurable set $B$, either $\mu(A \cap B) = 0$ or $\mu(A \setminus B) = 0$. Moreover, if $\mu$ is finite we can apply \cite[Theorem~2.1]{johnson} to decompose $\mu$ in a unique way as the sum of two finite measures $\nu$ and $\eta$ such that $\nu$ is atomless, $\eta$ is purely atomic and they are mutually singular (mutual singularity and mutual $S$-singularity are clearly equivalent on finite measures). This also allows the decomposition of $L_1(\mu)$ as 
\begin{equation}\label{eq:Maharam}
    L_1(\mu) \equiv L_1(\nu) \oplus_1 L_1(\eta) \equiv L_1(\nu)\oplus_1 \ell_1(I),
\end{equation} 
where $I$ is the set of all atoms for $\eta$ (up to a measure zero set).

Lastly, we will also need the fact that every $L_1(\mu)$ space can be decomposed as
\begin{equation}\label{eq:decompfinite}
    L_1(\mu) \equiv \ell_1(\Lambda,L_1(\mu_\lambda)),
\end{equation} 
where each $\mu_\lambda$ is finite (c.f. e.g. \cite[Remark P. 501]{deflo}). Furthermore, an inspection of the proof allows us to conclude that $\mu$ is atomless iff every $\mu_\lambda$ is also atomless.

\begin{example}
Let $X = L_1(\mu)$. By (\ref{eq:decompfinite}) and the stability of the SD2P by $\ell_1$-sums, $X$ has the SD2P iff each $L_1(\mu_\lambda)$ has it which, by (\ref{eq:Maharam}) and \cite[Theorem 3.4.4]{kmrzw25}, it happens iff each $\mu_\lambda$ is atomless, that is, iff $\mu$ is atomless. 

So, if $X$ has the SD2P, each $\mu_\lambda$ has no atoms and by \cite[Corollary 2.5]{rodrz23} $L_1(\mu_\lambda)$ must be a \textit{weakly almost square Banach space}, so a \textit{locally almost square Banach space} (see \cite{all16} for definitions). By \cite[Proposition 5.3]{all16} $X$ must also be a LASQ space.

Now, given a free ultrafilter $\U$ over $\N$, it follows that $X_\mathcal U$ is LASQ by \cite[Proposition 4.2]{hardtke18} and, in particular, $X_\mathcal U$ has the slice-D2P by \cite[Proposition 2.5]{k14}. On the other hand, we can obtain by \cite[Theorem 3.3]{hein80} that $X_\mathcal U$ is isometrically isomorphic to an $L_1$ space and, by the same argument as in the first paragraph, we obtain that its measure must be atomless. By \cite[Theorem 3.4.4]{kmrzw25} we conclude that $X_\U$ has the Daugavet property, so it has the SD2P.

We have then proved that if $X$ has the SD2P, then $X_\U$ also has it for every free ultrafilter $\U$ over $\N$, that is, we have proved that $X$ has the USD2P too.
\end{example}

\begin{example}
If $X$ is an $L_1$ predual, then $X$ has the SD2P iff $X$ it is infinite dimensional (for instance, observe that $X^{**}$ is an infinite dimensional $C(K)$ space \cite[Theorem 6 \S 21 and Theorem 6 \S 11]{elton}, so it satisfies the SD2P \cite[Corollary 3.5]{abg} and so does $X$ being an ai-ideal of $X^{**}$). As any ultrapower of an $L_1$ predual is also an $L_1$ predual (see \cite[Proposition 2.1]{hein81}), we also obtain that in this case the SD2P and its uniform version are equivalent.
\end{example}

\begin{example} If $X$ is an infinite-dimensional uniform algebra, then $X$ has the USD2P. This follows since $X_\mathcal U$ is a uniform algebra \cite[Proof of Theorem 7.6.14 c)]{kmrzw25} for any free ultrafilter $\U$ of $\N$ and because infinite-dimensional uniform algebras have the SD2P \cite[Theorem 4.2]{aln}.
\end{example}

Observe that in all these examples we are using that, either stronger properties than the SD2P have well known characterisations for its uniform versions, or that a class of Banach spaces is known to be stable under ultraproducts. To obtain more examples which, a priori, cannot be proved by the previous techniques, we will need a practical characterisation of the USD2P which does not involve ultrapowers. 

\subsection{Characterisation of spaces with the USD2P}

We will first introduce a geometrical characterisation of the SD2P which, as it does not require the access to its topological dual, is of intependent interest.

\begin{theorem}\label{carsd2p}
Let $X$ be a Banach space. The following are equivalent:
\begin{enumerate}
    \item $X$ has the \emph{SD2P}.
    \item For every $n \in \N$ and $\eps > 0$
    $$B_{\ell_\infty^n(X)} = \cco \left( C^{n, \eps}(X) \right).$$
    \item For every $n \in \N$ and $\eps > 0$
    $$B_{\ell_\infty^n(X)} \subset \cco \left( C^{n, \eps+}(X) \right).$$
\end{enumerate}
\end{theorem}

\begin{proof}
(1)$\Rightarrow$(2) Let $n \in \N$ and $\eps > 0$ be fixed.

As $C^{n,\eps}(X) \subset B_{\ell_\infty^n(X)}$ it is obvious that $\cco \left( C^{n, \eps}(X) \right) \subset B_{\ell_\infty^n(X)},$ so it only remains to prove the reciprocal inclusion.

Consider $\mathcal{W}$ to be a $w$-open neighbourhood basis of $0$ as a directed set under the reverse inclusion order. Now, given $x =(x_1, \cdots, x_n) \in B_{\ell_\infty^n(X)}$ and $U\in\mathcal W$, we have that
$$\sum_{i=1}^n \dfrac{1}{n} \left[ (x_i + U) \cap B_X \right]$$
is a mean of relatively weakly open subsets of $B_X$. By Theorem~\ref{theo:charSD2P2019} we can obtain, for every $i=1, \cdots, n$, a net $\{x_U^i\}_{U \in \mathcal{W}}$ such that $x_U^i \in (x_i + U) \cap B_X$ for every $U \in \mathcal{W}$, which implies that $\{ x_U^i\}_{U\in \mathcal W}\rightarrow^{w}  x_i$ and it is satisfied that 
\begin{align}\label{caracSD2P1}
    \dfrac{1}{n}\norm{\sum_{i=1}^nx_U^i} > 1-\eps
\end{align}
holds for every $U\in \mathcal W$. We can then define in $\ell_{\infty}^n(X)$ the net $\{x_U\}_{U \in \mathcal{W}}$ by $x_U=(x_U^1, \cdots, x_U^n)$ and observe that, because the weak topology on a finite product is the product of the weak topologies, $\{x_U\}_{U \in \mathcal{W}}\rightarrow^w x$. Then, by Mazur's theorem, we have that
$$x \in \overline{\{x_U: U \in \mathcal{W}\}}^w \subset \overline{\co\{x_U: U \in \mathcal{W}\}}^w = \cco\{x_U: U \in \mathcal{W}\},$$
so given any $\delta> 0$ there exist an $m \in \N$, some $\alpha_1, \cdots, \alpha_m >0$  with $\sum_{j=1}^m \alpha_j =1$ and $U_1, \cdots, U_m \in \mathcal{W}$ such that
\begin{align*}
    \norm{x - \sum_{j=1}^m \alpha_j x_{U_j}}_\infty < \delta.
\end{align*}
By (\ref{caracSD2P1}) we have that $\sum_{j=1}^m \alpha_j x_{U_j} \in \co\left( C^{n, \eps}(X) \right)$ and the arbitrariness of $\delta > 0$ allows us to conclude that $x \in \cco\left( C^{n, \eps}(X)\right)$ as wanted.
\vspace{1mm}

(2)$\Rightarrow$(3) This implication is trivial because $C^{n, \eps}(X) \subset C^{n,\eps+}(X)$.  

(3)$\Rightarrow$(1) Let
$$C=\sum_{i=1}^n \dfrac{1}{n} S(B_X, f_i, \alpha_i)$$
be a convex combination of slices of $B_X$ and $\eps >0$ be given. Without loss of generality we can assume that each $f_i \in S_{X^*}$ and that all the $\alpha_i$ are equal to some $0 < \alpha < 1$ since slices are decreasing with $\alpha$.
    
We can then take for every $i=1, \cdots, n$ an
$$x_i \in S\left(B_X, f_i, \left( \dfrac{\alpha}{n} \right)^2 \right)$$
and $0<r<\eps$ such that for every $i= 1, \cdots, n$
\begin{align}\label{caracSD2P4}
    f_i(x_i) > 1-\left( \dfrac{\alpha}{n} \right)^2  + r.
\end{align}
If we denote $x = (x_1, \cdots, x_n) \in B_{\ell_\infty^n(X)}$, by (3) we have that there exists an $m \in \N$, some $x^j \in C^{n,\delta+}(X)$\footnotemark for $j=1, \cdots, m$ and $\alpha_1, \cdots, \alpha_m >0$  with $\sum_{j=1}^m \alpha_j =1$ such that
\begin{align}\label{caracSD2P5}
    \norm{x - \sum_{j=1}^m \alpha_j x^j}_\infty < (1+\delta)r.
\end{align}
\footnotetext{Where
$$0 < \delta < \min\left\{\dfrac{\eps}{2},\min_{1 \leq i \leq n} \dfrac{f_i(x_i) - \left(1-\left( \dfrac{\alpha}{n} \right)^2   + r\right)}{1-\left( \dfrac{\alpha}{n} \right)^2  + r}\right\}.$$}
By (\ref{caracSD2P4}), (\ref{caracSD2P5}) and the choice of $\delta$ we then obtain that for every $i=1, \cdots, n$
\begin{align}\label{caracSD2P7}
    \dfrac{1}{1+\delta}\sum_{j=1}^m \alpha_j x_i^j \in S\left(B_X, f_i, \left( \dfrac{\alpha}{n} \right)^2 \right),
\end{align}
and then, that for every $i=1, \cdots, n$ there exists some $j_i \in \{1, \cdots, m\}$ such that
\begin{align}\label{caracSD2P7.5}
    \dfrac{1}{1+\delta} x_i^{j_i} \in S\left(B_X, f_i, \left( \dfrac{\alpha}{n} \right)^2 \right) \subset S\left(B_X, f_i, \dfrac{\alpha}{n} \right) \subset S(B_X, f_i,\alpha).
\end{align}
As $x^j \in C^{n,\delta+}(X)$ for every $j= 1, \cdots, m$,  we have that
\begin{align}\label{caracSD2P6}
    \dfrac{1}{n}\norm{\sum_{i=1}^n \dfrac{1}{1+\delta}x^j_i} > \dfrac{1-\delta}{1+\delta} > 1-\eps,
\end{align}
so if we can prove that there exists a common $j_0$ such that (\ref{caracSD2P7.5}) is satisfied for every $i=1, \cdots, n$, then $\sum_{i=1}^n \frac{1}{n} \frac{1}{1+\delta} x_i^{j_0} \in C$ and by (\ref{caracSD2P6}) and Theorem~\ref{theo:charSD2P2019} we would conclude that $X$ has the SD2P as we want.

To accomplish this it is enough to show that if we define for $i=1, \cdots, n$
$$G_i=\left\{j \in \{1, \cdots, m\}: \dfrac{1}{1+\delta}f_i(x_i^j) > 1-\dfrac{\alpha}{n}\right\},$$
(which are non-empty by (\ref{caracSD2P7.5})) we have that
\begin{align}\label{caracSD2P8}
    G=\bigcap_{i=1}^n G_i \not = \emptyset,
\end{align}
which will be proved if we show that for every $i=1, \cdots, n$
\begin{align}\label{caracSD2P9}
        \sum_{j \not \in G_i} \alpha_j < \dfrac{\alpha}{n}.
\end{align}
In fact, if (\ref{caracSD2P9}) is true then
$$\sum_{j \not \in G} \alpha_j \leq \sum_{i=1}^n \sum_{j \not \in G_i} \alpha_j < \sum_{i=1}^n \dfrac{\alpha}{n} = \alpha < 1,$$
and because $\sum_{j=1}^m \alpha_j = 1$ we conclude that $G \not = \emptyset$ as wanted.

To prove (\ref{caracSD2P9}) we just need to observe that by (\ref{caracSD2P7}) we have that for $i=1, \cdots, n$
\begin{align*}
    1-\left( \dfrac{\alpha}{n} \right)^2 & < \sum_{j \in G_i} \alpha_j f_i\left(\dfrac{1}{1+\delta} x_i^j\right) + \sum_{j \not \in G_i} \alpha_j f_i\left(\dfrac{1}{1+\delta} x_i^j\right)\\
     & \leq \sum_{j \in G_i} \alpha_j + \sum_{j \not \in G_i} \alpha_j \left( 1-\dfrac{\alpha}{n} \right)= 1-\dfrac{\alpha}{n} \sum_{j \not \in G_i} \alpha_j,
\end{align*}
and reordering this inequality we obtain (\ref{caracSD2P9}) and conclude the proof.
\end{proof}

With this new tool at our disposal we can give a general characterisation of ultraproducts having the SD2P that, in particular, will provide us with the desired characterisation of the USD2P.

\begin{theorem}\label{ultrasd2p}
Let $\{X_\lambda\}_{\lambda \in \Lambda}$ be a family of Banach spaces and $\U$ a countably incomplete ultrafilter over $\Lambda$. The following are equivalent:
\begin{enumerate}
    \item $(X_\lambda)_\U$ has the \emph{SD2P}.
    \item For every $n \in \N$ and $\eps, \delta > 0$ there exists an $m \in \N$ such that
    $$\{\lambda \in \Lambda: D_m^{n, \eps}(X_\lambda) < \delta\} \in \U.$$
    \item For every $n \in \N$ and $\eps, \delta > 0$ there exists an $m \in \N$ such that
    $$\{\lambda \in \Lambda: D_m^{n, \eps+}(X_\lambda) < \delta\} \in \U.$$
\end{enumerate}
\end{theorem}

\begin{proof}
Along the proof we will denote $X = (X_\lambda)_\U$.

(1) $\Rightarrow$ (2) By way of contradiction, let us suposse that this implication is not true, that is, we will assume that $X$ has the SD2P, but there exist some $n_0 \in \N$ and $\eps_0, \delta_0 >0$ such that for every $m \in \N$
\begin{align}\label{teo1.1}
D_m = \{\lambda \in \Lambda: D_m^{n_0, \eps_0}(X_\lambda) \geq \delta_0\} \in \U.
\end{align}

Then, if we can find some $x \in B_{\ell_\infty^{n_0}(X)}$ such that
\begin{align}\label{teo1.2}
    d \left(x,\co \left( C^{n_0, \eps_0}(X) \right) \right) \geq \dfrac{\delta_0}{2},
\end{align}
by Theorem \ref{carsd2p} we will obtain that $X$ does not have the SD2P, which is the desired contradiction. So, to complete the proof it is enough to use (\ref{teo1.1}) to construct a point $x \in B_{\ell_\infty^{n_0}(X)}$ such that for any $y \in \co \left( C^{n_0, \eps_0}(X) \right)$
\begin{align}\label{teo1.3}
\norm{x - y}_\infty = \max_{1 \leq i \leq n_0} \norm{x^i - y^i}_\U \geq \dfrac{\delta_0}{2}.
\end{align}

To accomplish our objective, consider first a family $\{A'_p\}_{p \in \N_0} \subset \U$ such that $\bigcap_{p \in \N_0} A'_p = \emptyset$ (which exists as $\U$ is countably incomplete) and define for $p \in \N$
\begin{align}
    A_p = \bigcap_{1 \leq n \leq p} (D_n \cap A_{n-1}') \in \U.
\end{align}
It is obvious that $A_{p+1} \subset A_p$ for every $p \in \N$ and that $\bigcap_{p \in \N} A_p = \emptyset$, so $\{A_p \setminus A_{p+1}: p \in \N\}$ is a partition of $\bigcup_{p \in \N} A_p$.

Given some $\lambda \in \Lambda$, if $\lambda \not \in \bigcup_{p \in \N} A_p$ we can define $x_\lambda = (0, \cdots, 0) \in B_{\ell_\infty^{n_0}(X_\lambda)}$, and in any other case there exists a unique $p \in \N$ such that $\lambda \in A_p \setminus A_{p+1}$, so in particular $\lambda \in D_p$ and
$$\sup_{z \in B_{\ell_\infty^{n_0}(X_\lambda)}} d \left( z, C_p^{n_0, \eps_0}(X_\lambda) \right) \geq \delta_0 
> \dfrac{\delta_0}{2}.$$
Consequently, there exists some $x_\lambda = (x^1_\lambda, \cdots, x^{n_0}_\lambda) \in B_{\ell_\infty^{n_0}(X_\lambda)}$ such that
\begin{align}\label{teo1.12}
    d \left( x_\lambda, C_p^{n_0, \eps_0}(X_\lambda) \right) > \dfrac{\delta_0}{2}.
\end{align}

Defining $x = \left( [x^1_\lambda]_\U, \cdots, [x_\lambda^{n_0}]_\U \right)$, it is clear that $x \in B_{\ell_\infty^{n_0}(X)}$ and it remains to see if it satisfies (\ref{teo1.3}).

Given any $y \in \co \left( C^{n_0, \eps_0}(X) \right)$, we have that $y \in C_m^{n_0, \eps_0}(X)$ for some $m \in \N$, so there exist some $\alpha_1, \cdots, \alpha_m > 0$ such that $\sum_{j=1}^m \alpha_j = 1$ and some $z^1, \cdots, z^m \in B_{\ell_\infty^{n_0}(X)}$ (so that we can express them as $z^j = \left( z^{j,1}, \cdots, z^{j,n_0} \right)$ with $\norm{z^{j,i}}_\U \leq 1$) such  that
\begin{align}
    & y = \sum_{j=1}^m \alpha_j z^j, \label{teo1.4}\\
    & \dfrac{1}{n_0}\norm{\sum_{i=1}^{n_0} z^{j,i}}_\U > 1-\eps_0.\label{teo1.5}
\end{align}

As for every $i =1, \cdots, n_0$ and $j=1, \cdots, m$ we have that $z^{j,i} \in X$, we can choose some representatives $(z^{j,i}_\lambda)_{\lambda \in \Lambda} \in \prod_{\lambda \in \Lambda} X_\lambda$ such that $z^{j,i} = [z^{j,i}_\lambda]_\U$. Observe that we can assume, with no loss of generality, that $\Vert z^{j,i}_\lambda\Vert\leq 1$ holds for every $\lambda$ in virtue of Remark~\ref{remark:elereprebolaunidad}. Using (\ref{teo1.4}) we obtain that if we define $(y^i_\lambda)_{\lambda \in \Lambda}$ as
\begin{align}\label{tt}
    y^i_\lambda = \sum_{j=1}^m \alpha_j z^{j,i}_\lambda,
\end{align}
then it is a representative for every $y^i$ with $i=1, \cdots, n$. Finally, we will denote $y_\lambda = (y^1_\lambda, \cdots, y^{n_0}_\lambda)$ and $z^j_\lambda = (z_\lambda^{j,1}, \cdots, z_\lambda^{j,n_0})$.

Now, utilizing that $\norm{x^i - y^i}_\U = \lim_\U \norm{x^i_\lambda - y^i_\lambda}$ we obtain that for every $\eta > 0$
\begin{align}\label{teo1.6}
    B_\eta = \bigcap_{i=1}^{n_0} \left\{ \lambda \in \Lambda: \abs{\norm{x^i_\lambda - y^i_\lambda} -  \norm{x^i - y^i}_\U} < \eta \right\} \in \U.
\end{align}
Also, by (\ref{teo1.5}) we have that
\begin{align}\label{teo1.7}
    C = \bigcap_{j=1}^m \left\{ \lambda \in \Lambda : \dfrac{1}{n_0} \norm{\sum_{i=1}^{n_0} z^{j,i}_\lambda} > 1 - \eps_0 \right\} \in \U.
\end{align}

Taking $\lambda \in A_m \cap B_\eta \cap C$ we have that for every $i=1, \cdots, n_0$
$$\norm{x^i - y^i}_\U > \norm{x^i_\lambda - y^i_\lambda} - \eta,$$
so
\begin{align}\label{teo1.9}
    \norm{x -y}_\infty > \norm{x_\lambda - y_\lambda}_\infty - \eta.
\end{align}
In addition, as $\lambda \in C $ we know that $z^{j}_\lambda \in C^{n_0, \eps_0}(X_\lambda)$ for every $j=1, \cdots, m$ and by (\ref{tt})
\begin{align}\label{teo1.10}
    y_\lambda = \sum_{j=1}^m \alpha_j z^j_\lambda \in C_m^{n_0, \eps_0}(X_\lambda).
\end{align}

Finally, as $\lambda \in A_m$ there exists a unique $p \geq m$ such that $\lambda \in A_p \setminus A_{p+1}$ and by (\ref{teo1.12}) we get that
\begin{align}\label{teo1.11}
d(x_\lambda,  C_m^{n_0, \eps_0}(X_\lambda)) \geq d(x_\lambda,  C_p^{n_0, \eps_0}(X_\lambda)) > \dfrac{\delta_0}{2}.
\end{align}
Putting together (\ref{teo1.9}), (\ref{teo1.10}) and (\ref{teo1.11}) we conclude that
$$\norm{x - y}_\infty > \norm{x_\lambda - y_\lambda}_\infty - \eta \geq d(x_\lambda,  C_m^{n_0, \eps_0}(X_\lambda)) - \eta > \dfrac{\delta_0}{2} - \eta.$$
By the arbitrariness of $\eta > 0$ we obtain that
$$\norm{x - y}_\infty \geq \dfrac{\delta_0}{2}$$
and the arbitrariness of $y \in \co \left( C^{n_0, \eps_0}(X) \right)$ proves (\ref{teo1.3}) as desired.

(2) $\Rightarrow$ (3) This is trivial taking into account that $D_m^{n, \eps+}(X_\lambda) \leq D_m^{n, \eps}(X_\lambda)$.

(3) $\Rightarrow$ (1) Because of Theorem \ref{carsd2p}, it is enough to show that given $n \in \N$, $\eps > 0$ and $x = (x^1, \cdots, x^n) \in B_{\ell_\infty^n (X)}$ we have that
$$x \in \cco \left( C^{n, \eps+}(X) \right).$$
In other words, we must see that for every $\delta > 0$
\begin{align}\label{teo2.1}
    B(x, \delta) \cap \co \left( C^{n, \eps+}(X) \right) \not = \emptyset.
\end{align}

As for every $i=1, \cdots, n$ we know that $x^i \in X$, there exist some functions $(x^i_\lambda)_{\lambda \in \Lambda} \in \prod_{\lambda \in \Lambda} X_\lambda$ such that $x^i = [x^i_\lambda]_\U$, and as $\lim_\U \norm{x^i_\lambda} = \norm{x^i}_\U \leq 1$ we can assume, in virtue of Remark~\ref{remark:elereprebolaunidad}, that $\Vert x_\lambda^i\Vert\leq 1$ holds for every $i$ and every $\lambda$.

Furthermore, by hypothesis there exists an $m \in \N$ such that
$$B = \left\{ \lambda \in \Lambda: D_m^{n, \eps/2+}(X_\lambda) < \delta \right\} \in \U,$$
so for any $\lambda \in B$ we have that
$$\sup_{z \in B_{\ell_\infty^n(X_\lambda)}} d\left( z, C_m^{n, \eps/2+}(X_\lambda) \right) < \delta.$$

In particular, if $\lambda \in B$
$$d\left( x_\lambda,  C_m^{n, \eps/2+}(X_\lambda) \right) < \delta$$
and there exists some $y_\lambda \in C_m^{n, \eps/2+}(X_\lambda)$ such that
\begin{align}\label{teo2.2}
     \norm{x_\lambda - y_\lambda}_\infty < \delta.
\end{align}
Since $y_\lambda \in C_m^{n, \eps/2+}(X_\lambda)$, there exist some $\alpha_\lambda^1, \cdots, \alpha_\lambda^m > 0$ satisfying $\sum_{j=1}^m \alpha_\lambda^j = 1$ and some $z_\lambda^{1}, \cdots, z_\lambda^{m} \in (1+\eps/2) B_{\ell_\infty^n(X_\lambda)}$ such that
\begin{align}\label{teo2.3}
y_\lambda = \sum_{j=1}^m \alpha_\lambda^j z_\lambda^{j}
\end{align}
and for every $j = 1, \cdots, m$
\begin{align}\label{teo2.4}
\dfrac{1}{n}\norm{\sum_{i=1}^n z_\lambda^{j, i}} > 1-\dfrac{\eps}{2}.
\end{align}

On the other hand, if $\lambda \not \in  B$ we simply define $\alpha_\lambda^j = 0, z_\lambda^j = (0, \cdots, 0)$ for every $j=1, \cdots, m$.

Now we can define, for every $j= 1, \cdots, m$,
$$\alpha^j = \lim_\U \alpha_\lambda^j,$$
which clearly satisfy that $\alpha^j \geq 0$ and $\sum_{j=1}^m \alpha^j = 1$. Define also
$$w^j = ([z_\lambda^{j,1}]_\U, \cdots, [z_\lambda^{j,n}]_\U) \in \ell_\infty^n(X), 1\leq j\leq m$$
and
$$w = \sum_{j=1}^m \alpha^j w^j \in \ell_\infty^n(X).$$
We also define for every $i=1, \cdots,n$ and $\lambda \in \Lambda$
$$w^i_\lambda = \sum_{j=1}^m \alpha^j z_\lambda^{j,i} \in \ell_\infty^n(X_\lambda)$$
and it is obvious that $w = ([w_\lambda^1]_\U, \cdots, [w_\lambda^n]_\U)$.

Let us see that
\begin{align}\label{teo2.5}
    w \in B(x, \delta) \cap \co \left( C^{n, \eps+}(X) \right),
\end{align}
which proves (\ref{teo2.1}) and finishes the proof.

On the one hand, we must see that for every $i=1, \cdots, n$
\begin{align}\label{teo2.6}
\norm{x^i - w^i}_\U \leq \delta.
\end{align}
Given $\eta > 0$ we know that
$$C_\eta = \left\{ \lambda \in \Lambda : \abs{\norm{x^i - w^i}_\U - \norm{x_\lambda^i - w_\lambda^i}} < \dfrac{\eta}{2} \right\} \in \U$$
and that
$$D = \bigcap_{j=1}^m \left\{ \lambda \in \Lambda: \abs{\alpha^j - \alpha_\lambda^j} < \dfrac{\eta}{2m(1+\frac{\eps}{2})}\right\} \in \U,$$
so taking $\lambda \in  B \cap C_\eta \cap D$ we obtain that
\begin{align*}
    \norm{x^i - w^i}_\U \underset{\lambda \in C_\eta}{\leq} & \norm{x_\lambda^i - w_\lambda^i} + \dfrac{\eta}{2} \\
    \leq & \norm{x^i_\lambda - \sum_{j=1}^m \alpha_\lambda^j z_\lambda^{j,i}} + \norm{\sum_{j=1}^m (\alpha_\lambda^j - \alpha^j) z_\lambda^{j,i}} + \dfrac{\eta}{2} \\ 
    \underset{(\ref{teo2.3})}{=} & \norm{x_\lambda^i - y_\lambda^i} + \norm{\sum_{j=1}^m (\alpha_\lambda^j - \alpha^j) z_\lambda^{j,i}} + \dfrac{\eta}{2} \\
    \underset{(\ref{teo2.2})}{\leq} & \delta + \left( 1 + \dfrac{\eps}{2} \right)\sum_{j=1}^m \abs{\alpha_\lambda^j - \alpha^j} + \dfrac{\eta}{2} \underset{\lambda \in D}{<} \delta + \eta.
\end{align*}
The arbitrariness of $\eta > 0$ proves (\ref{teo2.6}).

On the other hand, we must see that $w \in \co \left( C^{n, \eps+}(X) \right)$ and, since $w = \sum_{j=1}^m \alpha^j w^j$, it is enough to see that for any $j=1, \cdots, m$
\begin{align}\label{teo2.7}
    w^j \in C^{n, \eps+}(X).
\end{align}

In the first place, $w^j \in (1+\eps)B_{\ell_\infty^n(X)}$ because for every $i=1, \cdots, n$ and $\lambda \in \Lambda$ we have that $\norm{z^{j,i}_\lambda} \leq 1+\eps/2$ and then
$$\norm{w^{j, i}}_\U = \norm{[z_\lambda^{j,i}]}_\U = \lim_\U \norm{z^{j,i}_\lambda} \leq 1 + \dfrac{\eps}{2} < 1+ \eps.$$
On the other side, given $\eta > 0$ we know that
$$E_\eta = \left\{ \lambda \in \Lambda : \abs{ \norm{\sum_{i=1}^n w^{j,i}}_\U - \norm{\sum_{i=1}^n z^{j,i}_\lambda}} < \eta \right\} \in \U,$$
so taking any $\lambda \in  B \cap E_\eta$ we obtain by (\ref{teo2.4}) that
\begin{align*}
    \dfrac{1}{n} \norm{\sum_{i=1}^n w^{j,i}}_\U > \dfrac{1}{n} \norm{\sum_{i=1}^n z^{j,i}_\lambda} - \dfrac{1}{n}\eta > 1-\dfrac{\eps}{2} - \dfrac{1}{n}\eta.
\end{align*}
The arbitrariness of $\eta > 0$ implies that
$$\dfrac{1}{n} \norm{\sum_{i=1}^n w^{j,i}}_\U \geq 1-\dfrac{\eps}{2} > 1-\eps$$
and the proof of (\ref{teo2.7}) is complete.

Putting together (\ref{teo2.6}) and (\ref{teo2.7}) we have proved (\ref{teo2.5}) and we can conclude the whole proof.
\end{proof}

\begin{corollary}\label{charUSD2P}
Let $X$ be a Banach space. The following are equivalent:
\begin{enumerate}
    \item $X$ has the \emph{USD2P}.
    \item There exists a free ultrafilter $\U$ over $\N$ such that $X_\U$ has the \emph{SD2P}.
    \item There exists a set $\Lambda$ and a countably incomplete ultrafilter $\U$ over $\Lambda$ such that $X_\U$ has \emph{SD2P}.
    \item For every $n \in \N$ and $\eps > 0$
        $$\lim_{k \to \infty} D_k^{n, \eps}(X) = 0,$$
        that is, given any $\delta > 0$ there exists an $m \in \N$ such that $D_m^{n, \eps}(X) < \delta$. 
    \item For every $n \in \N$ and $\eps > 0$
        $$\lim_{k \to \infty} D_k^{n, \eps+}(X) = 0,$$
        that is, given any $\delta > 0$ there exists a $m \in \N$ such that $D_m^{n, \eps+}(X) < \delta$. 
    \item For any set $\Lambda$ and any countably incomplete ultrafilter $\U$ over $\Lambda$ the ultrapower $X_\U$ has the \emph{SD2P}.
\end{enumerate}
\end{corollary}

\begin{proof}
(1) $\Rightarrow$ (2), (2) $\Rightarrow$ (3) and (6) $\Rightarrow$ (1) are trivial given there exists a free ultrafilter over $\N$ and that any free ultrafilter over $\N$ is countably incomplete.

(3) $\Rightarrow$ (4) By (3) there exists a countably incomplete ultrafilter $\U$ over some set $\Lambda$ such that $X_\U$ has the SD2P. Let $n \in \N$ and $\eps > 0$ be fixed.

Given any $\delta >0$, by Theorem \ref{ultrasd2p} there exists an $m \in \N$ such that
$$\left\{ \lambda \in \Lambda : D_m^{n, \eps}(X) < \delta \right\} \in \U,$$
so, in particular, this last set is non-empty and we get that
$$D_m^{n, \eps}(X) < \delta.$$
Moreover, if $k \geq m$ we have that
$$D_k^{n, \eps}(X) \leq D_m^{n, \eps} (X) < \delta$$
and the arbitrariness of $\delta >0$ proves that $\lim_{k \to \infty} D_k^{n, \eps}(X) = 0$.

(4) $\Rightarrow$ (5) It is enough to observe that $0 \leq D_k^{n, \eps+}(X) \leq D_k^{n, \eps}(X)$

(5) $\Rightarrow$ (6) Let $\U$ be a countably incomplete ultrafilter over some set $\Lambda$. By Theorem \ref{ultrasd2p}, to see that $X_\U$ has the SD2P it is enough to prove that for any $n \in \N$ and $\eps, \delta > 0$ there exists some $m \in \N$ such that
$$\left\{ \lambda \in \Lambda: D_m^{n,\eps+}(X) < \delta \right\} \in \U,$$
that is, it is enough to see that there exists an $m \in \N$ such that $D_m^{n,\eps+}(X) < \delta$, but this is immediate from (5).
\end{proof}

\subsection{Stability results}

To finish this section we will establish some stability results for the USD2P which will also help us to providing new examples.

It is easy to see that for any ultrafilter $\U$ over some set and for $1 \leq p \leq \infty$ the map
$$\begin{aligned}
    T : X_\U \oplus_p Y_\U & \longrightarrow (X \oplus_p Y)_\U \\
    ([x_i]_\U,[y_i]_\U) & \longmapsto [(x_y, y_i)]_\U
\end{aligned}$$
is well defined, linear and surjective. Besides, using the properties of ultrafilter limits it is also easy to prove that this map is indeed an isometry. The next stability result is now immediate if one takes into account \cite[Proposition 4.6]{aln} and \cite[Proposition 3.1 and Theorem 3.2]{abg}:

\begin{proposition}\label{stabilitysum}
Let $X$ and $Y$ be Banach spaces. Then:
\begin{enumerate}
    \item $X \oplus_1 Y$ has the \emph{USD2P} iff $X$ and $Y$ also have it.
    \item If one of $X$ or $Y$ has the \emph{USD2P}, then $X \oplus_\infty Y$ also has it.
    \item For $1 < p < \infty$, $X \oplus_p Y$ does not have the \emph{USD2P}.
\end{enumerate} 
\end{proposition}

\begin{remark}
This proposition also allows us to distinguish the USD2P and a uniform version of the D2P. Indeed, it is enough to take any Banach space with this uniform version of the D2P (any of the examples we gave for the USD2P will suffice) and observe that $X \oplus_2 X$ will also have the uniform version of the D2P (by the isometry $T$ and \cite[Theorem 2.4]{abg}), but it will not have the USD2P by Proposition \ref{stabilitysum}. We do not know if this uniform version of the D2P is different from the uniform version of the slice-D2P studied in \cite{rz25}.
\end{remark}

Looking for stability under subspaces, we already mentioned that the SD2P is inherited by ai-ideals \cite[Proposition 3.3]{aln2} and the following result shows that the same holds true for the uniform version. 

\begin{proposition}\label{prop:herenciaSD2Puniaiideales}
Let $X$ be a Banach space with the \emph{USD2P} and let $Y\subseteq X$ be a subspace. If $Y$ is an ai-ideal in $X$, then $Y$ also has the \emph{USD2P}.
\end{proposition}

\begin{proof}
Since $Y$ is an ai-ideal in $X$, then $Y_\mathcal U$ is an ai-ideal in $X_\mathcal U$ for every free ultrafilter $\mathcal U$ over $\N$ \cite[Proposition 6.7]{mr25}. Consequently, given any free ultrafilter $\mathcal U$ over $\N$, we get that $Y_\mathcal U$ has the SD2P since $X_\mathcal U$ has it and this property is inherited by ai-ideals. Since $Y_\mathcal U$ has the SD2P for every free ultrafilter $\mathcal U$ over $\N$ we conclude that $Y$ has the USD2P, as desired.
\end{proof}

\begin{remark}\label{remark:sliced2punif}
The above techniques also allows to prove that the uniform slice-D2P is inherited by ai-ideals.
\end{remark}

\section{New examples}\label{sec:examples}

As we have mentioned in the introduction, the aim of this section is to exhibit new examples of Banach spaces with the USD2P where such property does not follow neither from a description of the ultrapower spaces nor from stronger properties and, consequently, where the application of Theorem~\ref{ultrasd2p} is critically needed. Since each of both examples require their own notation and their specific tools, we will split this section in two subsections.

\subsection{Spaces of Lipcshitz functions}

A \emph{pointed metric space} is just a metric space $M$ in which we distinguish element, denoted by $0$. Given a pointed metric space $M$, we write $\Lip(M)$ to denote the Banach space of all Lipschitz maps $f:M\longrightarrow \mathbb R$ which vanish at $0$, endowed with the Lipschitz norm defined by
$$ \| f \| := \sup\left\{\frac{f(x)-f(y)}{d(x,y)} \colon x,y\in M,\, x \neq y \right\}.$$
Observe that the above space is isometrically isomorphic to the quotient space of all the Lipschitz functions in $M$ over the subspace of constant functions under the seminorm given by the best Lipschitz constant (see \cite[Chapter 2]{weaver18} for details). We will denote by $\justLip(M)$ such quotient space.

Concerning the diameter two properties in ultraproducts of Lipschitz-free spaces, it is known that the Daugavet property and its uniform version are equivalent for spaces of Lipschitz functions. Moreover, if $M$ is a pointed metric space, then $\Lip(M)$ has the Daugavet property iff $M$ is a length space (see \cite[Theorem 3.5]{gpr} and the comments on pages 480-481). So, if $M$ is a length space this result implies that $\Lip(M)$ has the USD2P.

Also, in the case where $M$ is uniformly discrete and bounded the ultrapower of $\Lip(M)$ is again a space of Lipschitz functions \cite[Proposition 3.2]{gg}, so in this case it could be possible to study the USD2P via the study of the SD2P in Lipschitz spaces. 

These two examples are of the kind of the ones given in Section \ref{sec:USD2P}, but in the first one we have a complete characterisation for the Daugavet property in these spaces, so there is no more room for improvement; and in the second one \cite[Proposition 3.2]{gg} also suggests that bounded uniformly discrete metric spaces might be the only ones for which spaces of Lipschitz functions are stable under ultrapowers. Therefore, in order to study the USD2P in non-uniformly discrete or unbounded metric spaces we cannot use those approaches.

To attack that case, first observe that in \cite{iva06} it was proved that if a metric space $M$ is not uniformly discrete or unbounded, then $\Lip(M)$ has the slice-D2P. Later, in \cite{lanru2020} the result was improved to show that, indeed, $\Lip(M)$ must have the SD2P. We will now use the characterisation given in Corollary~\ref{charUSD2P} to prove that theses spaces also have the USD2P.

To accomplish our objective, it is enough to observe that with the same assumptions given in \cite[Lemma 2]{iva06}, one can conclude much more: 

\begin{lemma}\label{ivalemma}
Let  $(M, \rho)$ be a pointed metric space such that for every $\eps > 0$ there exist a couple of sequences $\{t_n\}_{n=1}^\infty, \{\tau_n\}_{n=1}^\infty \subset M$ and two sequences of real numbers $R_n > \rho_n > r_n > 0$ (where $\rho_n = \rho(t_n, \tau_n)$) satisfying that
\begin{align}
    & 0 < \dfrac{2 \rho_n}{R_n - \rho_n} \leq \eps, \\
    & 0 < \dfrac{2 r_n}{\rho_n - r_n} \leq \eps,
\end{align}
and that the family of rings $\{B(t_n, R_n) \setminus B(t_n, r_n)\}_{n=1}^\infty$ is pairwise disjoint. Then, the space $\Lip(M)$ has the \emph{USD2P}.
\end{lemma}

The proof follows step by step the one given by Y. Ivakhno in \cite[Lemma 2]{iva06}, but we will include it for completeness and to address the differences.

\begin{proof}
Since $\Lip(M)$ and $\text{Lip}(M)$ (with the quotient Lipschitz norm by the subspace $\mathcal{C}$ of constant functions) are linearly isometric, it is enough to prove the result for $X = \text{Lip}(M)$.

We will use Corollary~\ref{charUSD2P} (5), that is, fixed $n \in \N$ and $\eps, \delta>0$ we must prove that there exists some $k \in \N$ such that $D_k^{n, \eps+}(X) < \delta$, that is, that for any $z \in B_{\ell_\infty^n(X)}$
\begin{align}\label{lemaiva1}
    d(z, C_k^{n,\eps+}(X)) < \delta.
\end{align}

Let then $z = (z^1 + \mathcal{C}, \cdots, z^n + \mathcal{C})$ be an arbitrary element of $B_{\ell_\infty^n(X)}$ (so that $\norm{z^i}_L \leq 1$ for every $i=1, \cdots, n$). If we denote $N_m = (M \setminus B(t_m, R_m)) \cup B(t_m, r_m) \cup \{\tau_m\}$ we can define for every $m \in \N$ the functions $z_m: N_m \rightarrow \R^n$ given by
$$\begin{array}{ll}
    z_m (t) = (z^1(t), \cdots, z^n(t)) & \text{si } t \in N_m \setminus \{\tau_m\}, \\
    z_m(t) = (z^1(t_m)  + \rho_m, \cdots, z^n(t_m) + \rho_m) & \text{si } t = \tau_m.
\end{array}$$
We will fist show that for every $m \in \N$ and any $i=1, \cdots, n$ we have that
$$\norm{z_m^i}_{\text{Lip}(N_m)} \leq 1+\eps.$$

Indeed, if $t, s \in N_m$ are two distinct points we can distinguish three cases:
\begin{itemize}
    \item If $t$ and $s$ are both not equal to $\tau_m$, then
    \begin{align*}
        \dfrac{\abs{z_m^i(t) - z_m^i(s)}}{\rho(t,s)} = \dfrac{\abs{z^i(t) - z^i(s)}}{\rho(t,s)} \leq \norm{z^i}_L \leq 1 < 1+\eps.
    \end{align*}
    \item If any of the points equals $\tau_m$, let us say $t = \tau_m$, and $s \not \in B(t_m, R_m)$, then
    \begin{align*}
        \dfrac{\abs{z_m^i(t) - z_m^i(s)}}{\rho(t,s)} & = \dfrac{\abs{z^i(t_m) + \rho_m - z^i(s)}}{\rho(\tau_m,s)} \leq \dfrac{\norm{z^i}_L \rho(t_m, s) + \rho_m}{\rho(\tau_m,s)} \\
        & \leq \dfrac{\rho(\tau_m,s) + 2\rho_m}{\rho(\tau_m,s)} = 1 + \dfrac{2\rho_m}{\rho(\tau_m,s)} \\
        & \underset{s \not \in B(t_m, R_m)}{\leq} 1+ \dfrac{2\rho_m}{R_m - \rho_m} \leq 1+\eps
    \end{align*}
    \item If any of the points equals $\tau_m$, let us say $t = \tau_m$, and $s \in B(t_m, R_m)$, then
    \begin{align*}
        \dfrac{\abs{z_m^i(t) - z_m^i(s)}}{\rho(t,s)} & = \dfrac{\abs{z^i(t_m) + \rho_m - z^i(s)}}{\rho(\tau_m,s)} \leq \dfrac{\norm{z^i}_L \rho(t_m, s) + \rho_m}{\rho(\tau_m,s)} \\
        & \underset{s  \in B(t_m, r_m)}{\leq} \dfrac{r_m + \rho_m}{\rho_m - r_m} = 1+ \dfrac{2 r_m}{\rho_m - r_m} \leq 1+\eps
    \end{align*}
\end{itemize}

We can then extend every $z_m^i$ to a Lipschitz function in $M$ with norm less than or equal to $1+\eps$ that, by notation abuse, we will again denote by $z_m^i$. To finish the proof it is then enough to see that for every $m \in \N$
\begin{align}\label{ivafin1}
    (z_m^1 + \mathcal{C}, \cdots, z_m^n + \mathcal{C}) \in C^{n, \eps+}(X)
\end{align}
and that for every $i=1, \cdots, n$ and $k \in \N$
\begin{align}\label{ivafin2}
    \norm{z^i - \dfrac{1}{k}\sum_{j=1}^k z_j^i}_L \leq \dfrac{4+2\eps}{k}.
\end{align}

Certainly, if (\ref{ivafin1}) and (\ref{ivafin2}) are true, we can take $k \in \N$ such that $k > \frac{4+2\eps}{\delta}$ and we obtain that 
\begin{align*}
    d(z, C_k^{n,\eps+}(X)) & \leq \max_{1 \leq i \leq n}\norm{(z^i + \mathcal{C}) - \dfrac{1}{k}\sum_{j=1}^k (z_j^i + \mathcal{C})} = \max_{1 \leq i \leq n}\norm{z^i  - \dfrac{1}{k}\sum_{j=1}^k z_j^i}_L \\
    & < \dfrac{4+2\eps}{k} < \delta
\end{align*}
which, by the arbitrariness of $z$, proves (\ref{lemaiva1}) as wanted.

First, observe that as $\norm{z_m ^i + \mathcal{C}} = \norm{z_m^i}_L \leq 1+\eps$, it is clear that
$$\norm{(z_m^1 + \mathcal{C}, \cdots, z_m^n + \mathcal{C})}_\infty \leq (1+\eps).$$
Moreover,
\begin{align*}
    \dfrac{1}{n} \norm{\sum_{i=1}^n (z_m^i + \mathcal{C})} & =  \dfrac{1}{n} \norm{\sum_{i=1}^n z_m^i}_L \geq \dfrac{1}{n} \dfrac{\abs{\sum_{i=1}^n z_m^i(t_m) - \sum_{i=1}^n z_m^i(\tau_m)}}{\rho_k} \\
    & = \dfrac{1}{n}\dfrac{\abs{\sum_{i=1}^n (z^i(t_m) - z^i(t_m) + \rho_m)}}{\rho_m} = 1 > 1-\eps
\end{align*}
and (\ref{ivafin1}) is proved.

Finally, for any fixed $k \in \N$ and $i \in \{1, \cdots, n\}$ we have that given any distinct points $t, s \in M$, if we denote by
\begin{align*}
    D_{t,s} & =\dfrac{\abs{\left(z^i(t) -  \dfrac{1}{k}\displaystyle\sum_{j=1}^k z_j^i (t)\right) - \left(z^i(s) -  \dfrac{1}{k}\displaystyle\sum_{j=1}^k z_j^i (s)\right) }}{\rho(t,s)} \\
    & = \dfrac{1}{k} \dfrac{\abs{\displaystyle\sum_{j=1}^k (z^i(t) - z_j^i (t)) - \displaystyle\sum_{j=1}^k (z^i(s) - z_j^i (s))}}{\rho(t,s)}
\end{align*}
we have that:
\begin{itemize}
    \item If $t, s \not \in \bigcup_{j=1}^k B(t_j,R_j) \setminus B(t_j,r_j)$, then $t, s \in \bigcap_{j=1}^k (N_j \setminus \{\tau_j\})$ and so $z_j^i(t) = z^i(t)$ and $z_j^i(s) = z^i(s)$ for every $j=1, \cdots, k$. It is then obvious that $D_{t,s} = 0$.
    \item If any of the two points is not in $\bigcup_{j=1}^k B(t_j,R_j) \setminus B(t_j,r_j)$, but the other one is (let us suppose, without loss of generality, that $t$ is not in that set, but $s$ is), then, as the rings are pairwise disjoint, there exits a unique $j_s \in \{1, \cdots, k\}$ such that $s \in B(t_{j_s},R_{j_s}) \setminus B(t_{j_s},r_{j_s})$ and $s \not \in \bigcup_{j=1, j \not = j_s}^k B(t_j,R_j) \setminus B(t_j,r_j)$. Since, as before, we have that $z_j^i(t) = z^i(t)$ and $z_j^i(s) = z^i(s)$ for every $j=1, \cdots, k$ with $j \not = j_s$, we obtain that
    \begin{align*}
        D_{t,s} & = \dfrac{1}{k} \dfrac{\abs{(z^i(t) - z^i(s)) - (z_{j_s}^i(t) - z_{j_s}^i(s))}}{\rho(t,s)} \\
        & \leq \dfrac{\norm{z^i}_L + \norm{z_{j_s}^i}_L}{k} \leq \dfrac{2+\eps}{k} < \dfrac{4+2\eps}{k}.
    \end{align*}
    \item If $t,s \in \bigcup_{j=1}^k B(t_j,R_j) \setminus B(t_j,r_j)$, as the rings are pairwise disjoint we know that there exist some $j_t, j_s \in \{1, \cdots, k\}$ such that $z_j^i(t) = z^i(t)$ and $z_j^i(s) = z^i(s)$ for every $j=1, \cdots, k$ with $j \not = j_t, j_s$ and then
    \begin{align*}
        D_{t,s} & =  \dfrac{1}{k} \dfrac{\abs{2(z^i(t) - z^i(s)) - (z_{j_t}^i(t) - z_{j_t}^i(s)) - (z_{j_s}^i(t) - z_{j_s}^i(s))}}{\rho(t,s)} \\
        & \leq \dfrac{2\norm{z^i}_L + \norm{z_{j_t}^i}_L + \norm{z_{j_s}^i}_L}{k} \leq \dfrac{4+2\eps}{k}.
    \end{align*}
\end{itemize}
In any case, we obtain that $D_{t,s} \leq \frac{4+2\eps}{k}$ and the arbitrariness of $t, s \in M$ prove (\ref{ivafin2}) and so the lemma.
\end{proof}

Now, in \cite[Theorems 1 and 2]{iva06} it is proved that a non-uniformly discrete metric space or an unbounded one satisfies the assumptions of Lemma \ref{ivalemma}, so we obtain the following theorem:

\begin{theorem}
Let $(M, \rho)$ be a pointed metric space. Then:
\begin{enumerate}
    \item If $M$ is not uniformly discrete, that is, if $\inf\{\rho(t,s): t \not = s\} = 0$, then $\Lip(M)$ has the \emph{USD2P}.
    \item If $M$ is compact, then $\Lip(M)$ has the \emph{USD2P} iff $M$ is infinite.
    \item If $M$ is not bounded, then $\Lip(M)$ has the \emph{USD2P}.
\end{enumerate}
\end{theorem}

\subsection{Infinite-dimensional centralizer}
In this section we will study spaces with an infinite-dimensional centralizer, so as to provide more examples of spaces with the USD2P for which its study via analysing its ultrapowers seems really hard.

Along the section we consider Banach spaces over $\mathbb K=\mathbb R$ or $\mathbb K=\mathbb C$. Following the notation of \cite{br09}, a  \textit{multiplier on $X$} is a bounded linear operator $T:X\longrightarrow X$ satisfying that every extreme point of $B_{X^*}$ is an eigenvector of $T^*$. Given a multiplier $T$ on $X$, then for every  $p\in \ext{B_{X^*}}$ there exists a scalar $a_T(p)$ such that
$$p\circ T=T^*(p)=a_T(p)p.$$
We define the \textit{centralizer of $X$}\label{centralizer}, denoted by $Z(X)$, as the set of those multipliers $T$ for which there exists another multiplier $S$ such that 
$$a_T(p)=\overline{a_S(p)}$$
holds for all $p\in \ext{B_{X^*}}$. It is obvious that when $X$ is a real Banach space then $Z(X)$ coincides with the set of all multipliers on $X$.

Anyway, $Z(X)$ is a closed subalgebra of $L(X)$ which is isometrically isomorphic to $\mathcal C(K_X)$ for a certain compact Hausdorff space $K_X$ \cite[Proposition 3.10]{beh}. Moreover $X$ can be identified in a unique way as a function module whose base space is just $K_X$ and such that the elements of $Z(X)$ are precisely the operators of multiplication
by the elements of $\mathcal C(K_X)$ \cite[Theorems 4.14 and 4.16]{beh}.

Note that in \cite[Theorem 2.4]{br09} it is proved that if $Z(X)$ is infinite-dimensional, then $X$ has the D2P. This is our motivation to get the following proposition.

\begin{proposition}\label{prop:centrainfipuntoex}
Let $X$ be a Banach space such that $Z(X)$ is infinite-dimensional. If $\ext{B_X}\neq \emptyset$, then $X$ has the \emph{USD2P}.
\end{proposition}

\begin{proof}

By the above, $X$ can be seen as a functional module such that  $Z(X)$  is the set of operators of multiplication by the elements of $\mathcal C(K)$ for $K=K_{X}$. As $Z(X)$ is infinite-dimensional, $K$ must be infinite.

Consequently we can find a sequence of pairwise disjoint open sets on $K$, say $\{O_n\}_{n\in\mathbb N}$. Given $n\in\mathbb N$ select $t_n\in O_n$ and, by Urysohn lemma, we can find $f_n\in\mathcal C(K)$ such that
$$\begin{array}{ccc}
0\leq f_n\leq 1, & f_n(t_n)=1, & f_n|_{K\setminus O_n}=0
\end{array}.$$

Now, fix $k \in \N$ and $\eps > 0$, and take $z = (x_1,\ldots, x_k)\in B_{\ell_\infty^k(X)}$. Given $n\in\mathbb N$ define
$$z_n:=((1-f_n)x_1+f_ne,\ldots, (1-f_n)x_k+f_n e),$$
where $e\in \ext{B_X}$.
We claim that $z_n\in C^{k,\eps}(X)$.

On the one hand, given $1\leq i\leq k$ we get
\[\begin{split}\Vert (1-f_n)x_i+f_ne\Vert& =\sup_{t\in K}\Vert (1-f_n)(t)x_i(t)+f_n(t)e(t)\Vert\\
& \leq (1-f_n(t))\Vert x_i(t)\Vert+f_n(t)\Vert e(t)\Vert\leq 1,
\end{split}\]
from where we get that $z_n\in B_{\ell_\infty^k(X)}$.

On the other hand,
$$\frac{1}{k}\sum_{i=1}^k ((1-f_n)x_i+f_ne)=\frac{1}{k} (1-f_n) \sum_{i=1}^k x_i+f_n e$$
and
\[\begin{split}
\left\Vert (1-f_n)\frac{1}{k}\sum_{i=1}^k x_i+f_n e\right\Vert &  \geq \left\Vert (1-f_n)(t_n)\frac{1}{k}\sum_{i=1}^k x_i(t_n)+f_n(t_n) e(t_n)\right\Vert\\
& = \vert f_n(t_n)\vert \Vert e(t_n)\Vert=1 > 1-\eps
\end{split}\]
since $f_n(t_n)=1$ and $\Vert e(t)\Vert=1$ holds for every $t\in K$ \cite[Lemma 2.1]{br09}.

This proves that $z_n\in C^{k,\eps}(X)$ for every $n\in\mathbb N$, so
$\frac{1}{m}\sum_{j=1}^m z_j \in C_m^{k, \eps}(X)$ for every $m \in \N$. If we can prove that
\begin{align}\label{cent1}
\left\Vert z-\frac{1}{m}\sum_{j=1}^m z_j\right\Vert\leq \frac{2}{m},
\end{align}
by the arbitrariness of $z$ we obtain that
$$D_m^{k, \eps}(X) \leq \dfrac{2}{m},$$
so $\lim_{m \to \infty} D_m^{k, \eps}(X) = 0$ and Corollary~\ref{charUSD2P} allows us to finish the proof.

To prove (\ref{cent1}), note that
$$\left\Vert z-\frac{1}{m}\sum_{j=1}^m z_j\right\Vert=\sup_{t\in K} \left\Vert z(t)-\frac{1}{m}\sum_{j=1}^m z_j(t)\right\Vert,$$
and observe that if $t\in K\setminus\bigcup\limits_{j=1}^m O_j$, then
\[\begin{split}z_j(t)& =((1-f_j)(t)x_1(t)+f_j(t)e(t),\ldots, (1-f_j)(t)x_k(t)+f_j(t)e(t))\\
& =(x_1(t),\ldots x_k(t)).
\end{split}\]
This implies that $z(t)-\frac{1}{m}\sum_{j=1}^m z_j(t)=(0,0,\ldots, 0)$ for every $t\in K\setminus\bigcup\limits_{j=1}^m O_j$ and thus
\begin{align}\label{cent2}
\left\Vert z-\frac{1}{m}\sum_{j=1}^m z_j\right\Vert=\sup_{t\in K} \left\Vert z(t)-\frac{1}{m}\sum_{j=1}^m z_j(t)\right\Vert=\sup_{t\in \bigcup\limits_{j=1}^m O_j} \left\Vert z(t)-\frac{1}{m}\sum_{j=1}^m z_j(t)\right\Vert.
\end{align}
Now, if $t\in \bigcup\limits_{j=1}^m O_j$ then, as the open sets $O_n$'s are pairwise disjoint, $t\in O_{j_0}$ for a unique $j_0\in\{1,\ldots,m\}$. For $j\in \{1,\ldots, m\}\setminus \{j_0\}$ we have that $f_j(t)=0$, so $z_j(t)=(x_1(t),\ldots, x_k(t)) = z(t)$. Hence
\[
\begin{split}
z(t)-\frac{1}{m}\sum_{j=1}^m z_j(t)& =\frac{1}{m}\sum_{j=1}^m z(t)-z_j(t)\\
& =\frac{1}{m}\left(\sum\limits_{\substack{j=1\\j\neq j_0}}^m (z(t)-z_j(t))+z(t)-z_{j_0}(t) \right)\\
& =\frac{z(t)-z_{j_0}(t)}{m},
\end{split}
\]
and taking norm in the above we infer
\begin{align}\label{cent3}
\left\Vert z(t)-\frac{1}{m}\sum_{j=1}^m z_j(t) \right\Vert=\frac{\Vert z(t)-z_{j_0}(t)\Vert}{m}\leq \frac{2}{m}.
\end{align}
Since $t\in \bigcup\limits_{i=1}^m O_i$ was arbitrary, by (\ref{cent2}) and (\ref{cent3}) we conclude that
$$\left\Vert z-\frac{1}{m}\sum_{j=1}^m z_j\right\Vert\leq \frac{2}{m},$$
which proves (\ref{cent1}) as desired.
\end{proof}

Our next aim is to generalise the above result. In order to do so, once again, we borrow a piece of notation from \cite{br09}. Notice that, by the canonical isometric injection of a Banach space in its bidual, we have the following chain of Banach spaces
$$X\subseteq X^{**}\subseteq X^{(4}\subseteq\ldots\subseteq X^{(2n}\subseteq\ldots$$
Thus we have that $\bigcup\limits_{n=0}^\infty X^{(2n}$ is a vector space and, for a given $x\in \bigcup\limits_{n=0}^\infty X^{(2n}$, we can define
$$\Vert x\Vert:=\Vert x\Vert_{\min\{m \in 2\N:\, x \in X^{(m} \}}$$
defines a norm on $\bigcup\limits_{n=0}^\infty X^{(2n}$. We denote by $X^{(\infty}$ the completion of $\bigcup\limits_{n=0}^\infty X^{(2n}$ under the above norm.

It is known that if $T\in Z(X)$, then $T^{**}\in Z(X^{**})$ \cite[Corollary I.3.15]{hww}, so we can embed $Z(X)$ into $Z(X^{(\infty})$ in a natural way. Indeed, given $T\in Z(X)$, we can consider the action of $T$ on the elements of $X^{(2n}$ as the operator $T^{(2n}$ (the $2n$-th adjoint of $T$) for every $n\in\mathbb N$ and, under this point of view, we can actually see $T$ as an element of $Z(X^{(\infty})$ \cite[Proposition 4.3]{br09}.

Let us begin with an easy observation.

\begin{proposition}\label{prop:XaiidealXinfi}
Let $X$ be a Banach space. Then $X$ is an ai-ideal in $X^{(\infty}$.
\end{proposition}

\begin{proof}
As every finite-dimensional subspace of $\bigcup\limits_{n=0}^\infty X^{(2n}$ must be contained in some $X^{(2n}$ and $X$ is an ai-ideal of every $X^{(2n}$, we conclude that $X$ is an ai-ideal in $\bigcup\limits_{n=0}^\infty X^{(2n}$. The result then follows by \cite[Lemma 6.2]{mr25}.
\end{proof}

Now we can get the desired generalisation of Proposition~\ref{prop:centrainfipuntoex}.

\begin{theorem}\label{theo:centrainfiXinfy}
Let $X$ be a Banach space such that $X^{(\infty}$ has an infinite-dimensional centralizer. Then $X$ has the \emph{USD2P}.
\end{theorem}

\begin{proof}
If we denote $Y = X^{(\infty}$, by assumption $Z(Y)$ is infinite-dimensional, which implies that $Z(Y^{**})$ is also infinite-dimensional \cite[Corollary I.3.15]{hww}. By Proposition~\ref{prop:centrainfipuntoex} we get that $Y^{**}$ has the USD2P and, as $Y$ is an ai-ideal in $Y^{**}$, using Proposition \ref{prop:herenciaSD2Puniaiideales} we obtain that $Y$ has the USD2P. Since $X$ is an ai-ideal in $Y$ by Proposition~\ref{prop:XaiidealXinfi}, we deduce that $X$ has the USD2P by a new application of Proposition~\ref{prop:herenciaSD2Puniaiideales}.
\end{proof}
Let us conclude the section with some examples where $X^{(\infty}$ has an infinite-dimensional centralizer and, consequently, $X$ has the uniform SD2P.

\begin{example}\label{example:1centrainfi}
The following spaces $X$ satisfy that $X^{(\infty}$ has an infinite-dimensional centralizer:
\begin{enumerate}
\item Non-reflexive $JB^*$-triples \cite[Theorem 5.3]{br09}. In particular, so do infinite-dimensional $C^*$-algebras.
\item Every non-reflexive Banach space such that $X^*$ is $L$-embedded \cite[Proposition 3.3]{ab2010}.
\item $\mathcal C(K,X)$ for every infinite compact Hausdorff topological space $K$ and for every Banach space $X$ \cite[Proposition 3.2]{br09}.
\item $L(X,Y)$ for every Banach spaces $X$ and $Y$ satisfying that either $Z(X^*)$ or $Z(Y)$ are infinite-dimensional \cite[Lemma VI.1.1]{hww}.
\end{enumerate}
\end{example}

\section*{Acknowledgements}  

Theorem \ref{carsd2p} was obtained by the authors of the paper in collaboration with Gin\'es L\'opez-P\'erez, appearing in a preprint version of the paper \cite{lmr26} at arxiv.org on 2024/04/30 with reference arXiv:2404.19457. We express our deep gratitude to Gin\'es for permitting us to include the aforementioned theorem.

This research has been supported  by MCIU/AEI/FEDER/UE\\  Grant PID2021-122126NB-C31, by MICINN (Spain) Grant \\ CEX2020-001105-M (MCIU, AEI) and by Junta de Andaluc\'{\i}a Grant FQM-0185. The research of E. Mart\'inez Va\~n\'o has also been supported by Grant PRE2022-101438 funded by MCIN/AEI/10.13039/501100011033 and ESF+. The research of A. Rueda Zoca was also supported by Fundaci\'on S\'eneca: ACyT Regi\'on de Murcia grant 21955/PI/22.

\end{document}